\documentclass[a4paper,10pt]{article}
\usepackage{amssymb,latexsym,amsmath}
\usepackage{amsfonts}
\usepackage{undertilde}
\usepackage{enumitem}
\usepackage{amsthm}
\usepackage[retainorgcmds]{IEEEtrantools}
\usepackage{thmtools}
\usepackage{mathrsfs}

\declaretheorem[name=Definition,style=definition,qed=$\dashv$,
numberwithin=section]{dfn}
\declaretheorem[name=Example,style=definition,sibling=dfn]{exm}
\declaretheorem[name=Theorem,style=plain,sibling=dfn]{tm}
\declaretheorem[name=Lemma,style=plain,sibling=dfn]{lem}

\declaretheorem[name=Remark,style=definition,sibling=dfn]{rem}

\declaretheorem[name=Claim,style=plain,numbered=no]{clm*}
\declaretheorem[name=Sublaim,style=plain,numbered=no]{sclm*}
\declaretheorem[name=Case,style=definition]{case}
\declaretheorem[name=Subcase,style=definition]{scase}
\declaretheorem[name=Subsubcase,style=definition]{sscase}

\newcommand{\CC}{\mathbb C}
\newcommand{\sub}{\subseteq}
\newcommand{\cross}{\times}

\newcommand{\inter}{\cap}
\newcommand{\om}{\omega}

\newcommand{\OR}{\mathrm{OR}}
\newcommand{\Hull}{\mathrm{Hull}}
\newcommand{\cut}{\backslash}
\newcommand{\Tt}{\mathcal{T}}
\newcommand{\Uu}{\mathcal{U}}
\newcommand{\rg}{\mathrm{rg}}
\newcommand{\dom}{\mathrm{dom}}
\newcommand{\ins}{\trianglelefteq}

\newcommand{\pins}{\triangleleft}

\newcommand{\crit}{\mathrm{cr}}

\newcommand{\rest}{\!\upharpoonright\!}
\newcommand{\com}{\circ}

\newcommand{\lh}{\mathrm{lh}}
\newcommand{\Ult}{\mathrm{Ult}}
\newcommand{\sats}{\models}
\newcommand{\J}{\mathcal{J}}
\newcommand{\ZFC}{\mathsf{ZFC}}
\newcommand{\pistol}{\P}
\newcommand{\es}{\mathbb{E}}
\newcommand{\mubar}{{\bar{\mu}}}
\newcommand{\kappabar}{{\bar{\kappa}}}
\newcommand{\eps}{\varepsilon}

\newcommand{\pred}{\mathrm{pred}}
\newcommand{\un}{\cup}
\newcommand{\core}{\mathfrak{C}}
\newcommand{\id}{\mathrm{id}}
\newcommand{\sq}{\mathrm{sq}}
\newcommand{\conc}{\ \widehat{\ }\ }

\DeclareMathOperator{\card}{card}
\DeclareMathOperator{\cof}{cof}

\newcommand{\lpole}{\left\lfloor}
\newcommand{\rpole}{\right\rfloor}
\newcommand{\univ}[1]{\lpole #1\rpole}
\newcommand{\tu}{\textup}

\newcommand{\lex}{\mathrm{lex}}
\newcommand{\dropset}{\mathscr{D}}
\newcommand{\exit}{\mathrm{ex}}
\newcommand{\playerII}{\mathrm{II}}
\newcommand{\strength}{\mathrm{str}}

\newcommand{\bkgd}{\mathsf{b}}
\begin{document}
\title{Reconstructing resurrection}
\author{Farmer Schlutzenberg}
\maketitle
\begin{abstract}
Let $R$ be an iterable weak coarse premouse and let $N$ be a premouse with Mitchell-Steel 
indexing, produced by a fully backgrounded $L[\es]$-construction of $R$. We identify and 
correct a problem with the process of resurrection used in the proof of iterability of $N$.\footnote{Some time after writing this note,
the author found an alternate fix to the problem presented here, and also to the analogous problem with the copying constuction,
which preserves tree order between $\Tt$ and the lift tree $\Uu$, under the assumption that the proper segments of the premouse $M$ satisfy standard condensation facts
(and it is hard to imagine a situation in which we can't assume this).
The lift $E^\Uu_\alpha$ of $E^\Tt_\alpha$ in the alternate fix can, however, be slightly less related to $E^\Tt_\alpha$
than in the standard lifting procedure, which it seems might be a disadvantage. The alternate fix is to appear.}
\end{abstract}

\section{Introduction}
The purpose of this note is to fill a small gap in the proof of iterability 
of inner models with Mitchell-Steel indexing built by $L[\es]$-construction with full background 
extenders (see \cite[\S 12]{fsit}).
Such 
constructions are used to build canonical 
inner models of set theory having large cardinals. They are important, for example, in calibrating 
the strength of determinacy axioms against large cardinal axioms. Such constructions are ubiquitous 
in inner model theory. Iterability requires roughly that natural ultrapowers of the model 
be wellfounded. It is essential to know that $L[\es]$-constructions produce iterable inner 
models, when they indeed do; iterability helps ensure the canonicity of the inner models 
constructed, and it implies that they possess basic fine structural properties (such as 
condensation), which are needed for the general theory.

The inner models under consideration are of the form $L[\es]$,
where $\es$ is a sequence of partial extenders $E$.
Consider a model $L[\es]$ built by a full background extender construction.
The iterability proof of \cite{fsit} relies on the  
\emph{resurrection} process described there, which lifts elements $E$ of $\es$ to
extenders $E^\bkgd$ of $V$. The (main) gap we discuss lies in this process: the process can fail, 
and 
in fact, for some $E$ it seems there is no obvious candidate for $E^\bkgd$, if one requires 
that $E^\bkgd$ be an extender of $V$.\footnote{Our modified resurrection process will produce a 
candidate $E^\bkgd$, and this may or may not be an extender of $V$, but even if it is,
it may not be produced in the manner described in \cite{fsit}.} In Example \ref{exm:problem} we 
describe a specific situation, low in the large cardinal hierarchy, in which the resurrection 
process fails. The problem is closely related 
to the wrinkles in the copying construction described in 
\cite[pp. 1624--1625]{outline} and \cite{copy_con}, pertaining to type 3 premice.\footnote{In fact, 
because the 
iterability proof uses the copying construction, the issues with the copying construction themselves
also arise, so we incorporate the corrections described in \cite{copy_con} here, though for 
convenience we have arranged the bookkeeping differently.}

The main 
content of the paper is the proof of Theorem \ref{thm:resurrection}, an iterability proof for a 
model built by a full background extender construction, assuming iterability of the background 
universe; the point is that the iterability proof uses a correct resurrection process. 
The fix to 
resurrection is similar in nature to how the wrinkles in the copying construction are dealt with in
\cite{outline} and \cite{copy_con}, 
but there are more details. In the modified resurrection process, resurrection can itself involve 
taking (finitely many) ultrapowers, and $E^\bkgd$ can be an extender of an iterate of $V$, 
instead of $V$ itself.\footnote{The same problem arises in the iterability proof of \cite{cmip}, i.e. 
the 
proof of 
\cite[Theorem 9.14]{cmip}. The author believes that the fix 
we describe here can be adapted to that context. However, in that context there is more work to do,
particularly because the \emph{statement} of \cite[Theorem 9.14]{cmip} itself depends on the 
notion of resurrection, and so as we will see, does not literally make sense.}
%***change: removed comment below
%``The problem with resurrection only arises with premice with Mitchell-Steel 
%indexing (as opposed to 
%Jensen indexing)''.
%due to overall changes in background construction for Jensen indexed anyway.

Aside from the main gap, there also appears to be a small problem with the 
definition of \emph{weak $n$-embedding} (see \cite[pp. 52--53]{fsit} and 
\cite[Definition 4.1]{outline}); these 
embeddings arise 
naturally in the iterability proof (of \cite[\S12]{fsit}). This problem was noticed by Steve 
Jackson, and is explained in \cite[\S2]{copy_con}. (The problem is just potential, in that we do 
not know of an explicit example which contradicts any standard theorems regarding weak 
$n$-embeddings.) To deal with this, we take \emph{weak $n$-embedding} 
to be defined as in \cite[\S2]{copy_con}. We won't discuss this issue any further here.

\subsection{Conventions and Notation}
Given a transitive structure $R=(M,\ldots)$ with universe $M$, we write $\univ{R}$ for $M$,
and $\J(R)$ for the rud closure of $R\un\{R\}$.
If $\kappa<\OR^R$ and $\card^R(\kappa)$ is the largest cardinal of $R$, $(\kappa^+)^R$ denotes 
$\OR^R$.

We take \emph{premouse} as defined in \cite{outline}; in particular, they have Mitchell-Steel 
indexing.
Let $P$ be a premouse.
Given $\alpha\leq\OR^P$, we write $P|\alpha$ for the 
initial segment of $P$ of ordinal height $\alpha$, and $P||\alpha$ for its passive counterpart. We 
write $F^P$ for the active extender of $P$, $\es^P=\es(P)$ for the extender sequence of $P$, 
excluding $F^P$, and $\es_+^P=\es_+(P)$ for $\es^P\conc F^P$. Given a short extender $E$, we write 
$\crit(E)$ for the critical point of $E$, $\nu(E)$ for the natural length of $E$, 
$\lh(E)$ for the length of the trivial completion of $E$ and $\strength(E)$ for the strength of 
$E$. Let $\Tt$ be an iteration tree. If $\lambda+1<\lh(\Tt)$ we write 
$\exit^\Tt_\lambda=M^\Tt_\lambda|\lh(E^\Tt_\lambda)$ (\emph{ex} for 
\emph{exit extender}).
%***dropped M^\Tt_\infty etc as it was unused

Given premice 
$P,Q$, and a 
fine structural embedding $\pi:P\to Q$, the phrase ``$\pi:P\to Q$'' conventionally indicates that, 
literally, $\dom(\pi)=\core_0(P)$. Recall that for type 3 premice $P$, $P^\sq$ denotes the 
squash of $P$, and has universe $\univ{P|\nu(F^P)}$ (and a predicate coding 
$F^P\rest\nu(F^P)$); see 
\cite[\S3]{fsit}. When $P$ is type 3, 
$\core_0(P)=P^\sq$, so embeddings $\pi$ as above do not act,
at least not directly, on elements of $P\cut\core_0(P)$. It seems that this convention probably 
helped to disguise one of the problems with which we deal here. From now on in this 
paper we 
display all 
domains and codomains literally, writing, for example, $\pi:\core_0(P)\to\core_0(Q)$, so as to keep 
the true domain of $\pi$ in mind. (However, we do use the convention that fine structural 
notions such as $\rho_1^P$, and fine structural ultrapowers, are literally computed over 
$\core_0(P)$.)

We take \emph{weak $n$-embedding} to be defined as in \cite[\S2]{copy_con}.

Other notation and terminology is standard and mostly follows \cite{outline}.

\section{Resurrection}\label{sec:resurrection}

We first define a fairly general kind of full background extender construction (\emph{nice} 
construction), which includes 
typical full background extender constructions in the literature. Then in 
\ref{exm:problem} below, we give a specific example of the problem with resurrection. After 
this we will sketch the fix to this problem, and then, in the proof of Theorem 
\ref{thm:resurrection}, give a (more or less) complete iterability proof incorporating the fix.
\begin{dfn}\label{dfn:fm}
 Let $M$ be an active premouse and $\kappa<\OR^M$. We say that $\kappa$ is \emph{finely measurable} 
in $M$ iff $\kappa=\crit(E)$ for some $M$-total extender $E\in\es_+^M$.
\end{dfn}

\begin{dfn}\label{dfn:nice_con}
A \textbf{weak-coarse-premouse (wcpm)} is a \emph{premouse} as defined in
\cite[Definition 1.1]{it}.\footnote{Of course we are using \emph{premouse} differently here. We do 
not use the term \emph{coarse premouse} because this is used differently 
in \cite{cmip}.}

Suppose $V=(\univ{V},\delta)$ is a wcpm.\footnote{This hypothesis just means that we work inside 
some wcpm. It is not intended to imply that $V\sats\ZFC$.}
For $\lambda\leq\delta+1$, a \textbf{nice construction (of length $\lambda$)} is a sequence 
$\left<N_\alpha\right>_{\alpha<\lambda}$ such that
\begin{enumerate}[label=\tu{(}\roman*\tu{)}]
 \item for each $\alpha<\lambda$, $N_\alpha$ is a premouse,
 \item $N_0=V_\om$,\footnote{Although we restrict to pure premice 
$N_\alpha$ here, this is not important; everything in the paper relativizes immediately to 
premice above some fixed set.}
 \item for each limit $\gamma<\lambda$, $N_\gamma=\liminf_{\alpha<\gamma}N_\alpha$, and
 \item for each $\alpha+1<\lambda$, $N_\alpha$ is 
$\om$-solid and either $N_{\alpha+1}=\J(\core_\om(N_\alpha))$ or there are $E,E^\bkgd$ such that
\begin{enumerate}
\item $N_{\alpha+1}=(N_\alpha,E)$,
\item $E^\bkgd\in V_\delta$ is an extender,
 \item $E\rest\nu(E)\sub E^\bkgd$,
 \item for each $\kappa<\nu(E)$, if $\kappa$ is finely measurable in $\Ult(N_\alpha,E)$ then 
$\kappa<\strength(E^\bkgd)$.\qedhere
\end{enumerate}
\end{enumerate}
\end{dfn}

\begin{rem}\label{rem:construction_facts}
 Suppose $V$ is a wcpm and let $\CC=\left<N_\alpha\right>_{\alpha<\lambda}$ be a nice construction. 
Recall the following basic facts from \cite{fsit}, which we will use freely. Let 
$\alpha<\lambda$.

Let $\rho\leq\gamma<\OR^{N_\alpha}$ be such 
that $\rho$ is a cardinal of $N_\alpha$ and $\rho_\om(N_\alpha|\gamma)=\rho$. Then there is a 
unique $\xi<\alpha$ such that $\core_0(N_\alpha|\gamma)=\core_\om(N_\xi)$.

Let $E\in\es_+(N_\alpha)$ be such that $E$ is total over $N_\alpha$. Then $\crit(E)$ is 
measurable (in $V$).
\end{rem}

\begin{dfn}
Let $M,N$ be premice of the same type and let $\pi:\core_0(M)\to\core_0(N)$ be an 
$\Sigma_0$-elementary embedding. 
We define the embedding $\psi_\pi$ as follows. If $M$ is passive then $\psi_\pi=\pi$. Otherwise,
\[ \psi_\pi:\Ult(\core_0(M),F^M)\to\Ult(\core_0(N),F^N) \]
is the embedding induced by the Shift Lemma. Note that in all cases, $\pi\sub\psi_\pi$.
\end{dfn}
The following lemma is easy to see, by considering the ISC and $\psi_\pi$:
\begin{lem}\label{lem:pres_fm}
 Let $M,N$ be active premice, let $\pi:\core_0(M)\to\core_0(N)$ be a weak $0$-embedding and 
$\kappa\in\OR^{\core_0(M)}$. Suppose that $\Ult(M,F^M)$ and $\Ult(N,F^N)$ are wellfounded.
 Then $\kappa$ is finely measurable in $\Ult(M,F^M)$ iff $\kappa$ is finely measurable in 
$\Ult(\core_0(M),F^M)$ iff $\pi(\kappa)$ is finely measurable in
$\Ult(N,F^N)$.
\end{lem}

\begin{exm}\label{exm:problem}
We now give an explicit example of the problem with resurrection, and sketch the fix we 
will use.

Suppose $V$ is a wcpm and $\CC=\left<N_\alpha\right>_{\alpha<\lambda}$ is a nice 
construction, every $\core_n(N_\alpha)$ is fully 
iterable and there is $\alpha$ such that $N_\alpha$ has a type 3 proper segment $M$ such 
that $M\sats$``$\cof(\nu(F^M))$ is measurable''. Let $\alpha$ be least such and $M$ the least such 
proper segment of $N_\alpha$.

We claim that $\rho_1^M=\om$ and $p_1^M=\emptyset$. For
let $H$ be the premouse such that $H^\sq$ is the transitive collapse of
$\Hull_1^{M^\sq}(\emptyset)$. Then $\rho_1^H=\om$, $H$ is $1$-sound,
and $0$-iterable, so a comparison shows that $H\ins M$. So it suffices to see that $H\sats$
``$\cof(\nu(F^H))$ is measurable''.
Note that $M\sats$``$\nu(F^M)=\kappa^{+\mu}$'', where $\kappa=\crit(F^M)$ and $\mu$ is the 
least measurable of $M$, and that if $\pi:\core_0(H)\to\core_0(M)$ is the uncollapse map,
then $\mu,\kappa\in\rg(\pi)$. Moreover, for each $\alpha<\mu$,
we have $\alpha\in\mu\inter\rg(\pi)$ iff $(\kappa^{+\alpha})^M\in\rg(\pi)$, since
the identity of $(\kappa^{+\alpha})^M$ is coded into a segment of $F^M$ in a $\Sigma_0$ way. 
Therefore $H\sats$``$\nu(F^H)=\kappabar^{+\mubar}$'',
where $\pi(\mubar,\kappabar)=\mu,\kappa$, which suffices.

It follows that $\alpha=\xi+1$ for some $\xi$, and letting $N=N_\xi$,
$\core_0(M)=\core_1(N)=\core_\om(N)$, and $N$ is 
active and type 3. Note that for all $0$-maximal iteration trees on $M$,
if $\alpha+1<\beta+1<\lh(\Tt)$ then either $\crit(E^\Tt_\alpha)=\crit(E^\Tt_\beta)$
or $\nu(E^\Tt_\alpha)<\crit(E^\Tt_\beta)$. So iterability at this level is very simple -- in particular,
for limit length $\Tt$ there is always exactly one $\Tt$-cofinal branch. Now there is 
a successor length $0$-maximal tree $\Tt$ on $M$ such that $N=M^\Tt_\infty$ and $b^\Tt$ does not 
drop in model. This can be seen in 
two ways: either because $M$ is 
below $0^\pistol$, or by the stationarity of $L[\es]$-constructions \cite[\S 
4]{scales_hybrid}. Moreover, the core embedding
$\upsilon:\core_0(M)\to\core_0(N)$
is $\upsilon=i^\Tt$.

Let $\nu=\nu(F^M)=\OR(M^\sq)$; and $\mu,\kappa$ were defined above. So
$\mu<\kappa$ and $\cof^M(\nu)=\mu$ and $i^\Tt(\kappa)=\crit(F^N)$ and  
$i^\Tt(\mu)$ is the least measurable
of $N$. So by \ref{rem:construction_facts}, $i^\Tt(\mu)$ is measurable in $V$,
so $\mu<i^\Tt(\mu)$, so $\mu=\crit(i^\Tt)$.
Let $P=M^\Tt_1$. By the preceding remarks, 
$P=\Ult_0(M,U)$ (recall that this means that $P^\sq=\Ult(M^\sq,U)$) where 
$U\in\es^M$ is the 
normal measure on $\mu$, and $1\in b^\Tt$ and $\deg^\Tt(1)=0$.

We claim that $\nu(F^N)<\psi_{i^\Tt}(\nu)=\psi_\upsilon(\nu)$. This follows from \cite[Lemma 
2.11]{copy_con}, but here things are simpler, so we include the proof for self-containment. In $M$, 
let $f:\kappa\to\kappa$ be $f(\alpha)=\alpha^{+\mu}$. Then $\nu=[\{\kappa\},f]^M_{F^M}$. Let 
$j=i^\Tt_{0,1}$ (so $j:M^\sq\to P^\sq$). Then
\[ \nu(F^P)=\sup 
j``\nu=(j(\kappa)^{+\mu})^P<(j(\kappa)^{+j(\mu)})^{\Ult(P,F^P)}=[\{j(\kappa)\},j(f)]^P_{F^P}.
\]
Since also $\nu(F^N)=\sup i^\Tt``\nu$, the claim follows easily.
Therefore, since $\nu$ is a limit cardinal of $M$, we have ($\dagger$) 
$N||\OR^N\pins\psi_\upsilon(M|\nu)$.

Now the resurrection maps of \cite{fsit} are formed by composing core 
embeddings. In particular, if $M|\lambda\pins M$
is active, and we wish to resurrect this to find some backgrounded 
ancestral extender,
then according to \cite{fsit}, we should consider $\upsilon(M|\lambda)$,
then resurrect this structure with further core embeddings, as needed. But if 
$\nu<\lambda<\OR^M$, the first step here does not make
literal sense, since $M|\lambda\notin\dom(\upsilon)$.
Moreover, we can't correct this by lifting $M|\lambda$ with $\psi_\upsilon$, since by 
($\dagger$) we have $N||\OR(N)\pins Q$ where $Q=\psi_\upsilon(M|\lambda)$, and so standard facts 
about nice constructions show that $Q$ was never constructed by $\CC$. So the usual 
resurrection process, applied to $M|\lambda$, breaks down.\end{exm}

To solve this problem, in the proof of \ref{thm:resurrection} below, we will use approximately
the following 
approach. It 
is similar to how the wrinkles in the copying construction are dealt with in \cite{outline} and 
\cite{copy_con}. We continue with
the scenario above. Let $E^\bkgd$ witness
\ref{dfn:nice_con}
with respect to $N$. Then (see the proof to follow) there is $\tilde{Q}$ 
in $\Ult(V,E^\bkgd)$ and 
an elementary embedding $M|\lambda\to\tilde{Q}$, such that $\tilde{Q}$ \emph{is} constructed by 
$i_{E^\bkgd}(\CC)$. Thus, 
we can move into $\Ult(V,E^\bkgd)$ and continue the resurrection process there. In the example 
above, the same issue will not arise again (because of the minimality of $M$), but in the more 
general case it could. In the latter 
case we take another ultrapower, and so on. This 
procedure will terminate in finitely many steps, yielding a successful resurrection.
We next give a detailed iterability proof incorporating the fix to resurrection 
sketched 
here (or a slight variant thereof).\footnote{We will not actually define an explicit resurrection 
process, 
but instead fold the 
details directly into the iterability proof. Moreover, because we need to use background extenders 
$E^\bkgd$ as above in order to produce some form of resurrection, and there need not be a canonical 
choice of such $E^\bkgd$, there need not be a \emph{canonical} resurrection for a structure 
such as $M|\lambda$ above. However, using the methods to follow, it is easy enough to formulate an 
abstract notion of \emph{a resurrection} (of some initial segment of a model produced by a nice 
construction computed in a wcpm $R$), in terms of finite iteration trees on $R$.}

Usually one deals with $k$-maximal trees, or stacks thereof. However, it does not take 
much more work to give the iterability proof for the following more general class of trees (which 
includes stacks of $k$-maximal trees, and more), so it seems worthwhile doing so:

\begin{dfn}
Let $M$ be a $k$-sound premouse. We say that $\Tt$ is a \textbf{standard degree $k$ iteration tree 
on $M$} iff $\Tt$ satisfies the conditions described in \cite[\S5]{fsit}, 
except that we drop condition (3)
(that is, the condition ``$\alpha<\beta\implies\lh(E_\alpha)<\lh(E_\beta)$''), and strengthen the 
other 
clauses as follows. Let $M_\alpha=M^\Tt_\alpha$ and $E_\alpha=E^\Tt_\alpha$ and 
$\exit_\alpha=\exit^\Tt_\alpha$. For 
$\alpha\leq\beta<\lh(\Tt)$ and $\kappa<\OR^{M_\beta}$ we say that $[\alpha,\beta)$ is 
\textbf{$\kappa$-valid (for $\Tt$)} iff either $\alpha=\beta$, or
\[ 
(\kappa<\min_{\gamma\in[\alpha,\beta)}\nu(E^\Tt_\gamma))\ 
\&\ (\kappa^+)^{\exit_\alpha}=(\kappa^+)^{M_\beta}.\]
For 
$E\in\es_+^{M_\beta}$ we say that $[\alpha,\beta)$ is \textbf{$E$-valid (for $\Tt$)} iff 
$[\alpha,\beta)$ is $\crit(E)$-valid and either $\alpha=\beta$ or $E$ is 
$M_\beta$-total. We require 
that if $\pred^\Tt(\beta+1)=\alpha$ then $[\alpha,\beta)$ is $E_\beta$-valid. We also 
require that $M^{*\Tt}_{\beta+1}$ and $\deg^\Tt(\beta+1)$ be chosen as for $k$-maximal trees.

For an ordinal $\alpha$, we say that $M$ is \textbf{standardly $(k,\alpha)$-iterable} iff there is 
a winning strategy for player $\playerII$ in the iteration game on $M$ for standard degree 
$k$ trees of length at most $\alpha$.

A \textbf{putative standard tree} $\Tt$ is defined in terms of standard trees as usual
(that is, we make the same requirements except that if $\Tt$ has a last model then we do not 
require that it is wellfounded).
\end{dfn}

\begin{rem}\label{rem:standard}
We make a couple of remarks on standard iteration trees. See \cite[pp. 3--5]{jonsson} for more 
discussion; standard trees all meet the definition of \emph{iteration tree} used in 
\cite{jonsson}.
Let $\Tt$ be standard. If $\kappa<\OR^{M_\beta}$ then $[\beta,\beta)$ is trivially 
$\kappa$-valid. Let $\alpha<\beta$ and $\kappa$ be such that $[\alpha,\beta)$ is $\kappa$-valid. 
Then either $(\kappa^+)^{M_\beta}<\lh(E_\alpha)$ or $E_\alpha$ is type 2 and 
$(\kappa^+)^{M_\beta}=\lh(E_\alpha)$. For all $\gamma\in(\alpha,\beta)$, 
$(\kappa^+)^{M_\gamma}=(\kappa^+)^{M_\beta}<\lh(E_\gamma)$. Suppose further that
$\kappa=\crit(E)$ for some $M_\beta$-total extender $E\in\es_+^{M_\beta}$. Let $G$ be the type 1 
initial segment of $E$. Then $G\in\es^{M_\gamma}$ for all $\gamma\in(\alpha,\beta]$, and $G\in
\es(\Ult(\exit_\alpha,E_\alpha))$, and in particular, $\kappa$ is finely measurable in 
$\Ult(\exit_\alpha,E_\alpha)$ (assuming the latter is wellfounded).
\end{rem}
\begin{tm}\label{thm:resurrection}
Let $\theta\geq\om_1$ be a cardinal. Let $R$ be a wcpm\footnote{It is not particularly important 
that $R$ be a wcpm. We just need that iteration maps on $R$ for trees based on $V_{\delta^R}^R$ are 
sufficiently elementary, but we leave it to the reader to reduce the hypotheses.} and 
$\Sigma_R$ a \tu{(}partial\tu{)}
$(\theta+1)$-iteration strategy for $R$. Let $\CC\in R$ be such that $R\sats$``$\CC$ is a nice 
construction''. Let $\zeta<\lh(\CC)$ and $z\leq\om$.

If $\Sigma_R$ is total then $N_\zeta$ is $z$-solid and 
$\core_z(N_\zeta)$ is 
standardly $(z,\theta+1)$-iterable.

If $\Sigma_R$ is defined on all stacks of non-overlapping trees, then
$N_\zeta$ is $z$-solid and $\core_z(N_\zeta)$ is $(z,\cof(\theta),\theta+1)^*$-iterable.
\footnote{Recall that in the $(z,\mu,\theta+1)^*$-iteration game,
if in a single round a normal tree of length $\theta+1$ is produced,
with wellfounded models, then the entire game finishes and player II wins.}

%***change: added * to iterability (maybe remind reader)

If $\Sigma_R$ is defined on all non-overlapping trees, and if 
$N_\zeta$ is $z$-solid, then $\core_z(N_\zeta)$ is $(z,\theta+1)$-iterable.
\end{tm}

\begin{proof}
We mostly follow the usual proof, lifting trees on $N=\core_z(N_\zeta)$ to $R$ via copying and 
resurrection, but make modifications to deal with the problem described in \ref{exm:problem}.

Assuming that $\Sigma_R$ is defined on all trees of length $\leq\theta$, we will describe a 
strategy $\Sigma_N$ for player $\playerII$ in the 
standard $(z,\theta+1)$-iteration game 
on $N$. Let $\Tt$ be a putative standard degree $z$ tree on $N$ which is via $\Sigma_N$. 
Then by induction, we can lift $\Tt$ to a tree $\Uu=\pi\Tt$ on $R$ ($\Uu$ is to be 
defined), via $\Sigma_R$, and if $\Tt$ has limit length, use $\Sigma_R(\Uu)$ to define 
$\Sigma_N(\Tt)$. In particular, $\Tt$ has wellfounded models. At the end we will make some 
modifications to the construction which will ensure 
that $\Uu$ is non-overlapping if $\Tt$ is $z$-maximal.

We will have $\lh(\Uu)\geq\lh(\Tt)$, but in general may have $\lh(\Uu)>\lh(\Tt)$.
We index the nodes of $\Uu$ with 
elements of $\OR\cross\om$, which we order lexicographically.
For each node $\alpha$ of $\Tt$, $(\alpha,0)$ will be a node of $\Uu$, and the model
$M^\Uu_{\alpha 0}$ will 
correspond directly to $M^\Tt_\alpha$. However, there may also be a further \emph{finite} set of 
integers $i$ such that $(\alpha,i)$ is a node of $\Uu$, and then $M^\Uu_{\alpha i}$ is associated 
to a proper segment of $M^\Tt_\alpha$.
So if $\lh(\Tt)>1$ then the set of nodes of $\Uu$ will not be an 
initial segment of $\OR\cross\om$. For notational convenience we allow $\Uu$ to use padding. If 
$E=E^\Uu_{\alpha i}=\emptyset$ we 
consider $\strength^{M^\Uu_{\alpha i}}(E)=\OR(M^\Uu_{\alpha i})$.
Whether or not $E^\Uu_{\alpha i}=\emptyset$, we allow $\pred^\Uu(\beta,j)=(\alpha,i)$ for 
$(\beta,j)>(\alpha,i)$.

The following definition is a coarser variant of the notion of \emph{dropdown sequence} used in 
\cite[\S12]{fsit}; in a dropdown sequence one also records the various projecta 
$\rho_n^{M|\eta_i}$ in the 
interval 
$(\rho_\om^{M|\eta_i},\rho_\om^{M|\eta_{i-1}})$. At this stage we ignore these intermediate 
projecta. In the end we will index partial resurrection maps by potential critical points $\kappa$, 
not by projecta.

\begin{dfn}
 Let $M$ be a $k$-sound premouse and $\gamma\leq\OR^M$. Let $\rho(\OR^M)=\rho_k^M$ and 
for $\eta<\OR^M$ let $\rho(\eta)=\rho_\om(M|\eta)$. The 
\textbf{$(\gamma,k)$-model-dropdown 
sequence} of $M$ is the sequence $\sigma=\left<(\eta_i,\varrho_i)\right>_{i\leq n}$ of maximum 
length such 
that $\eta_0=\gamma$, and for each $i\leq n$, $\varrho_i=\rho(\eta_i)$, and if $i<n$ then 
$\eta_{i+1}$ is the least $\eta\in(\eta_i,\OR^M]$ such that $\rho(\eta)<\varrho_i$.
If $(\OR^M,\rho(\OR^M))\in\sigma$ then let $\tau=\emptyset$; otherwise let 
$\tau=\left<(\OR^M,\rho(\OR^M))\right>$.
The \textbf{extended $(\gamma,k)$-model-dropdown sequence} of $M$ is 
$\sigma\conc\tau$.\footnote{Note that these notions depend on $k$ as 
$\rho(\OR^M)=\rho_k^M$.}

Let $\sigma=(\sigma_0,\ldots,\sigma_{n-1})$ be a sequence. The \textbf{reverse} of 
$\sigma$ is 
$(\sigma_{n-1},\ldots,\sigma_0)$. If each $\sigma_i=(a_i,b_i)$ then 
$p_0[\sigma]=(a_0,\ldots,a_{n-1})$.
\end{dfn}

We now fix some notation and state some intentions. For $\alpha<\lh(\Tt)$ we write 
$M_\alpha=M^\Tt_\alpha$, $m_\alpha=\deg^\Tt(\alpha)$. For 
$\alpha+1<\lh(\Tt)$ we write $E_\alpha=E^\Tt_\alpha$, $\exit_\alpha=\exit^\Tt_\alpha$,
$M^*_{\alpha+1}=M^{*\Tt}_{\alpha+1}$. Fix $\alpha+1<\lh(\Tt)$. Let $\sigma$ be the 
extended $(\lh(E_\alpha),m_\alpha)$-model-dropdown sequence of $M_\alpha$ and let 
$\tilde{\sigma}$ be its reverse.
Let $u_\alpha+1=\lh(\sigma)$ and let $\left<\gamma_{\alpha i}\right>_{i\leq 
u_\alpha}=p_0[\tilde{\sigma}]$. So $\gamma_{\alpha 0}=\OR^{M_\alpha}$ and $\gamma_{\alpha 
u_\alpha}=\lh(E_\alpha)$.
Fix $i\leq u_\alpha$. Let $n_{\alpha i}=m_\alpha$ if $i=0$ and $n_{\alpha i}=\om$ otherwise.
Let $\CC_{\alpha i}=i^\Uu_{00,\alpha i}(\CC)$ and $\Delta_{\alpha i}=\lh(\CC_{\alpha i})$.
We will define $\xi_{\alpha i}<\Delta_{\alpha i}$, and letting 
$Q_{\alpha i}=N_{\xi_{\alpha i}}^{\CC_{\alpha i}}$, will define a weak
$n_{\alpha i}$-embedding 
\[ \pi_{\alpha i}:\core_0(M_\alpha|\gamma_{\alpha i})\to\core_{n_{\alpha i}}(Q_{\alpha i}). \]
For $m\leq n\leq n_{\alpha i}$ let
\[ \tau^{n m}_{\alpha i}:\core_n(Q_{\alpha i})\to\core_m(Q_{\alpha i}) \]
be the core embedding. Let $Q^\exit_\alpha=Q_{\alpha u_\alpha}$ and
\[ \pi^\exit_\alpha:\core_0(\exit_\alpha)\to\core_0(Q^\exit_\alpha), \]
where letting $n=n_{\alpha  u_\alpha}$,
\[ \pi^\exit_\alpha=\tau^{n 0}_{\alpha u_\alpha}\com\pi_{\alpha u_\alpha}.\]
Let 
$c_\alpha$ be the set of inaccessible cardinals $\kappa$ of $\exit_\alpha$ such that 
$\kappa<\nu(E_\alpha)$. Fix $\kappa\in c_{\alpha}$. Let $(\gamma,n_{\alpha \kappa})$ be the 
lexicographically least $(\gamma,n)$ such that $\gamma\geq\lh(E_\alpha)$ and either (i) 
$\gamma=\OR^{M_\alpha}$ and 
$n=m_\alpha$ or (ii) $\rho_{n+1}(M_\alpha|\gamma)\leq\kappa$. Note that $\gamma\in 
p_0[\tilde{\sigma}]$. 
Let $i=u_{\alpha \kappa}$ be such that $\gamma=\gamma_{\alpha i}$. We also define the
weak $n_{\alpha \kappa}$-embedding
\[ 
\pi_{\alpha \kappa}:\core_0(M^\Tt_\alpha|\gamma_{\alpha i})\to\core_{n_{\alpha \kappa}
}(Q_{\alpha i})\]
by $\pi_{\alpha \kappa}=\tau^{n m}_{\alpha i}\com\pi_{\alpha i}$,
where $n=n_{\alpha i}$ and $m=n_{\alpha \kappa}$.

If $\lh(\Tt)=\beta+1$ then $(\beta,0)$ will be the last node in $\Uu$, and we will also define 
$\xi_{\beta 0}<\Delta_{\beta 0}$, and letting 
$Q_{\beta 0}=N_{\xi_{\beta 0}}^{\CC_{\beta 0}}$, will define a
weak $m_\beta$-embedding
\[ \pi_{\beta 0}:\core_0(M^\Tt_\beta)\to\core_{m_\beta}(Q_{\beta 0}). \]

We will maintain the following conditions by induction on $\lh(\Tt)$:
\begin{enumerate}
\item\label{item:E_alpha,u_alpha_non-empty} For all $\alpha+1<\lh(\Tt)$, $E^\Uu_{\alpha 
u_\alpha}\neq\emptyset$.
\item\label{item:agreement} Let $\alpha+1<\lh(\Tt)$ and
$\kappa\in c_\alpha$. Let $\bar{\pi}=\pi_{\alpha\kappa}\rest(\kappa^+)^{\exit_\alpha}$. Then
\[ \bar{\pi}\sub 
\pi^\exit_\alpha\sub\pi_{\alpha+1,0}. \]
Suppose $\alpha<\delta<\lh(\Tt)$ and $[\alpha,\delta)$ is $\kappa$-valid. Then
\[ \bar{\pi}\sub\pi_{\delta 0}. \]
Given $(\eps,l,m)$ such that (i)
$(\alpha,u_{\alpha\kappa})\leq_\lex(\eps,l)\leq_\lex(\delta,0)$ and
(ii) $m\leq n=n_{\eps l}$ and (iii) if $(\alpha,u_{\alpha\kappa})=(\eps,l)$ then $m\leq 
n_{\alpha\kappa}$, we also have
$\bar{\pi}\sub\pi_{\eps\kappa}$ and  
$\bar{\pi}\sub\tau_{\eps l}^{nm}\com\pi_{\eps l}$.
\item\label{item:strength} Let $\alpha+1<\lh(\Tt)$ and
$U=\Ult(\exit_\alpha,E_\alpha)$ and $\kappa\in c_\alpha$.
Then $U$ is wellfounded. Suppose that $\kappa$ is finely measurable in 
$U$. Then 
$\pi_{\alpha \kappa}(\kappa)<\strength^{M^\Uu_{\alpha i}}(E^\Uu_{\alpha i})$
for all $i\in[u_{\alpha \kappa},u_\alpha]$.

\item\label{item:pred^U} Suppose $\alpha=\pred^\Tt(\beta)$ and let $i\leq u_\alpha$ be 
such that $M^*_{\beta}=M_\alpha|\gamma_{\alpha i}$
(so $\beta\in\dropset^\Tt$ iff $i\neq 0$). Then 
$(\alpha,i)=\pred^\Uu(\beta,0)$.

\item\label{item:drops} Suppose $(\alpha,i)=\pred^\Uu(\beta,j)$.
If $j\neq 0$ then
$\xi_{\beta j}< i^\Uu_{\alpha i,\beta j}(\xi_{\alpha i})$.
Suppose $j=0$; so $\alpha=\pred^\Tt(\beta)$. Then
$\xi_{\beta 0}= i^\Uu_{\alpha i,\beta 0}(\xi_{\alpha i})$
and letting $n=m_\beta$,
\[ \pi_{\beta 0}
\com i^{*\Tt}_{\beta}=
i^\Uu_{\alpha i,\beta 0}\com\tau_{\alpha i}^{n_{\alpha i} n}\com\pi_{\alpha i}.\]
\item\label{item:limits} Let $\lambda<\lh(\Tt)$ be a limit and let $\alpha<_\Tt\lambda$ be 
such that $(\alpha,\lambda]_\Tt$ does not drop in model. Then for all $\beta,i$, 
$(\alpha,0)\leq_\Uu(\beta,i)\leq_\Uu(\lambda,0)$ iff $i=0$ and 
$\alpha\leq_\Tt\beta\leq_\Tt\lambda$. Moreover, letting $m=m_\alpha$ and 
$n=m_\lambda$,
\[ \pi_{\lambda 0}\com 
i^\Tt_{\alpha,\lambda}=i^\Uu_{\alpha 0,\lambda 0}\com\tau^{m n}_{\alpha 0}\com\pi_{\alpha 0}. \]
\end{enumerate}

We set $\pi_{00}=\id$. The inductive hypotheses are trivial for $\Tt\rest 1$ and 
$\Uu\rest(0,1)$.

Now let $\lambda$ be a limit ordinal and suppose that the inductive hypotheses hold 
of $\Tt\rest\lambda$ and 
$\Uu\rest(\lambda,0)$; we will define $\Uu\rest(\lambda,1)$ and show that they hold for 
$\Tt\rest\lambda+1$ and $\Uu\rest(\lambda,1)$.

By hypothesis \ref{item:E_alpha,u_alpha_non-empty}, $\Uu\rest(\lambda,0)$ has limit 
length and is cofinally non-padded. Let 
$c=\Sigma_R(\Uu\rest(\lambda,0))$. Let $b=\Sigma_M(\Tt\rest\lambda)$ be the unique branch such that 
for 
eventually all $\alpha\in b$, we have $(\alpha,0)\in c$. By conditions
\ref{item:pred^U}--\ref{item:limits}, $b$ is indeed a well-defined $\Tt\rest\lambda$-cofinal 
branch, there are only finitely many drops in model along $b$, and there is a unique choice for 
$\pi_{\lambda 0}$ maintaining the commutativity (and all other) requirements.

Now let $\lambda=\delta+1$ and suppose that the inductive hypotheses hold for 
$\Tt\rest\delta+1$ and $\Uu\rest(\delta,1)$. We will show that they hold for $\Tt\rest\delta+2$ and 
$\Uu\rest(\delta+1,1)$.

\begin{case}\label{case:E_delta_is_active_ext} $E_\delta=F(M^\Tt_\delta)$.

We just give a sketch in this case as the details are mostly standard here, and anyway they are 
simpler than the next case. We have $u_\delta=0$. Set $E^{\Uu}_{\delta 0}$ to be some 
$E^\bkgd\in M^\Uu_{\delta 0}$ witnessing
\ref{dfn:nice_con} with respect to $Q_{\delta 0}$ (in $M^\Uu_{\delta 0}$, regarding $\CC_{\delta 
0}$). Let 
$\kappa=\crit(E_\delta)$ and 
$\alpha=\pred^\Tt(\delta+1)$ 
and $i=u_{\alpha\kappa}$. Note that 
$M^*_{\delta+1}=M_\alpha|\gamma_{\alpha i}$ and $n_{\alpha \kappa}=m_{\delta+1}$.
We have $\crit(E^\Uu_{\delta 0})=\pi^\exit_\delta(\kappa)$.

We claim that it 
is possible to set $\pred^\Uu(\delta+1,0)=(\alpha,i)$ (and we do set it so, as required by 
condition \ref{item:pred^U}). To see this we need to see that for every $(\eps,l)$ such that 
$(\alpha,i)\leq_\lex(\eps,l)<_\lex(\delta, 0)$,
we have $\crit(E^\Uu_{\delta 0})<\strength^{M^\Uu_{\eps l}}(E^\Uu_{\eps l})$. So suppose 
$\alpha<\delta$
and let $(\alpha,i)\leq(\eps,l)<(\delta,0)$.
Then $\kappa<\rho_{m_\delta}^{M_\delta}$ and
because $[\alpha,\delta)$ is 
$\kappa$-valid, condition \ref{item:agreement} gives
\begin{equation}\label{eqn:crit} \crit(E^\Uu_{\delta 0})=\pi^\exit_\delta(\kappa)=\pi_{\delta 
0}(\kappa)=\pi_{\alpha \kappa}(\kappa)=\pi_{\eps 
\kappa}(\kappa). \end{equation}
Let $G$ be the normal measure segment of 
$E_\delta$. Then $G\in\es(\Ult(\exit_\eps,E_\eps))$.
And because $[\alpha,\delta)$ is $\kappa$-valid, it is straightforward to see that if 
$\alpha<\eps$ then $u_{\eps\kappa}=0$.
So by condition \ref{item:strength} and line 
(\ref{eqn:crit}),
$\crit(E^\Uu_{\delta 0})<\strength^{M^\Uu_{\eps l}}(E^\Uu_{\eps l})$,
as required.

Now $\xi_{\delta+1,0}$ is determined by 
condition \ref{item:drops}, and we can go on to define $\pi_{\delta+1,0}$ as usual.
Standard calculations show that the 
inductive hypotheses are maintained (condition \ref{item:strength} for $\delta$ uses 
\ref{dfn:nice_con} and \ref{lem:pres_fm}).\end{case}

\begin{case} $E_\delta\neq F(M^\Tt_\delta)$.

In this case we must deal with the problem described in \ref{exm:problem}. Let 
$\sigma,\tilde{\sigma}$ 
be defined as before, with $\alpha=\delta$. Let $(\gamma,\rho)=(\tilde{\sigma})_1$. So $\gamma$ is 
the 
largest element of $p_0[\sigma]$ excluding $\OR^{M_\delta}$, and 
$\rho=\rho_\om^{M_\delta|\gamma}$ is a cardinal of $M_\delta$.

\begin{scase} $\rho\in\core_0(M_\delta)$.

Set $E^\Uu_{\delta 0}=\emptyset$; so $\pred^\Uu(\delta,1)=(\delta,0)$ and $M^\Uu_{\delta 
1}=M^\Uu_{\delta 0}$ and $i^\Uu_{\delta 0,\delta 1}=\id$.
Let
\[ \varphi:\core_0(M_\delta)\to\core_0(Q_{\delta 0}) \] 
be 
$\varphi=\tau^{m_\delta 0}_{\delta 0}\com\pi_{\delta 0}$. Then $\varphi(\rho)$ is a cardinal of 
$\core_0(Q_{\delta 0})$ and $\rho_\om(\varphi(M_\delta|\gamma))=\varphi(\rho)$. So by 
\ref{rem:construction_facts} we can let 
$\xi_{\delta 1}$ be the unique $\xi<\xi_{\delta 0}$ such that 
$\core_0(\varphi(M_\delta|\gamma))=\core_\om(N_\xi^{\CC_{\delta 1}})$, and let 
$\pi_{\delta 1}=\varphi\rest\core_0(M_\delta|\gamma)$.
\end{scase}

\begin{scase} $\rho\notin\core_0(M_\delta)$.

So $M_\delta$ is active type 3 and $\rho=\nu(F^{M_\delta})=\rho_0^{M_\delta}$. Let
$\upsilon:\core_0(M_\delta)\to\core_0(Q_{\delta 0})$ be 
$\upsilon=\tau^{m_\delta 0}_{\delta 0}\com\pi_{\delta 0}$. Let $\psi=\psi_\upsilon$.

\begin{sscase} $\psi(\rho)\leq\nu(F^{Q_{\delta 0}})$.

Proceed as in Subcase 1, using $\varphi=\psi$; in particular, $E^\Uu_{\delta 
0}=\emptyset$.
\end{sscase}

\begin{sscase}\label{sscase:special} $\psi(\rho)>\nu(F^{Q_{\delta 0}})$.

Here we need to do something different because $\psi(M_\delta|\gamma)$ is not constructed in 
$\CC_{\delta 0}$. Let $E^\Uu_{\delta 0}=E^\bkgd$ witness \ref{dfn:nice_con} with respect to 
$Q_{\delta 0}$. Let 
$F=F^{M_\delta}$ and $\kappa=\crit(F)$. Let $\Tt'$ be the putative standard tree on $M$ 
of the form $(\Tt\rest\delta+1)\conc\left<F\right>$, with $\alpha=\pred^{\Tt'}(\delta+1)$ as small 
as 
possible. Then $M_\delta|\gamma\pins M^{\Tt'}_{\delta+1}$. Let $i=u_{\alpha\kappa}$ and 
$n=n_{\alpha \kappa}$. Let $\pred^\Uu(\delta,1)=(\alpha,i)$; as in Case 1 this is possible. Let 
$Q'=i^{M^\Uu_{\alpha i}}_{E^\bkgd}(Q_{\alpha i})$ and
\[ \varphi:\core_0(M^{\Tt'}_{\delta+1})\to\core_n(Q')\]
be given as usual. Then $\rho<\rho_n(M^{\Tt'}_{\delta+1})$ so
$\varphi(\rho)<\rho_n(\core_n(Q'))$, and 
$\varphi(\rho)$ is a cardinal of $\core_n(Q')$ and 
$\varphi(\rho)=\rho_\om(\varphi(M_\delta|\gamma))$. 
Let $\xi_{\delta 1}$ be the unique $\xi<i^\Uu_{\alpha i,\delta 1}(\xi_{\alpha i})$ such that 
$\core_0(\varphi(M_\delta|\gamma))=\core_\om(N_\xi^{\CC_{\delta 1}})$. Let 
$\pi_{\delta 1}=\varphi\rest\core_0(M_\delta|\gamma)$.
\end{sscase}\end{scase}

This completes the definition of $\Uu\rest(\delta,2)$ in all subcases.

Suppose 
$\lh(E_\delta)=\gamma$.
Then $u_\delta=1$ and we let 
$E^\Uu_{\delta 1}$ witness \ref{dfn:nice_con} for $Q_{\delta 1}$.
We also have $\pi^\exit_\delta=\tau^{\om 0}_{\delta 1}\com\pi_{\delta 1}$.
Let $\kappa=\crit(E_\delta)$, so $\crit(E^\Uu_{\delta 
1})=\pi^\exit_\delta(\kappa)=\pi_{\delta\kappa}(\kappa)$.

Suppose $E^\Uu_{\delta 0}=\emptyset$. Let $\alpha=\pred^\Tt(\delta+1)$. We claim that we can set 
$\pred^\Uu(\delta+1,0)$ as required by condition \ref{item:pred^U}. For suppose $\alpha<\delta$.
Because $[\alpha,\delta)$ is $E_\delta$-valid, we have $\kappa<\rho_{m_\delta}^{M_\delta}$ and
\[ (\kappa^+)^{\exit_\alpha}=(\kappa^+)^{M_\delta}<\lh(E_\delta)=\gamma,\]
so $\kappa<\rho$.
So $\pi_{\delta 0}(\kappa)=\varphi(\kappa)=\pi_{\delta 1}(\kappa)$
and
$\varphi(\kappa)<\varphi(\rho)=\rho_\om(\varphi(M_\delta|\gamma))\leq
\crit(\tau^{\om 0}_{\delta 1})$,
so $\varphi(\kappa)<\crit(\tau^{\om 
0}_{\delta 1})$, so $\pi^\exit_\delta(\kappa)=\pi_{\delta 0}(\kappa)$.
The rest is much as in Case \ref{case:E_delta_is_active_ext}.

Now suppose $E^\Uu_{\delta 0}\neq\emptyset$, so Subsubcase \ref{sscase:special} attained;
in particular, $M_\delta$ is active type 3.
%Change: it earlier said ``so $M_\delta$ is active type 3'' later; I moved this comment up here where it belongs.
Note that 
$\crit(\tau^{\om 0}_{\delta 1})\geq\nu(F^{Q_{\delta 0}})$, so $\pi_{\delta\mu}$ and $\pi^\exit_\delta$ agree 
appropriately for each $\mu$. Suppose we want 
to set $\pred^\Uu(\delta+1,0)<(\delta,1)$. So
$\kappa<\nu(F^{M_\delta})=\rho<\gamma$, so $E_\delta$ is 
$M_\delta$-total, and
\[ \crit(E^\Uu_{\delta 1})=\pi_{\delta\kappa}(\kappa)=\upsilon(\kappa)<\nu(F^{Q_{\delta 0}}),\]
where $\upsilon$ is as above. 
So $\upsilon(\kappa)$ is finely measurable in $Q_{\delta 0}$,
so $\upsilon(\kappa)<\strength^{M^\Uu_{\delta 0}}(E^\Uu_{\delta 0})$.
So much as in Case 1, we can 
set $\pred^\Uu(\delta+1,0)$ appropriately. To see that condition \ref{item:strength} holds in this 
case, let $U=\Ult(\exit_\delta,E_\delta)$ and $\mu\in 
c_\delta$ and suppose that $0=u_{\delta\mu}$
and $E\in\es_+^U$ witnesses that $\mu$ is finely measurable in $U$. Then 
$\mu<\rho$ (as $0=u_{\delta\mu}$) and therefore the normal measure segment 
$G$ of $E$ is in $\es^{\exit_\delta}$, so $G$ witnesses that $\mu$ is 
finely measurable in $\Ult(M_\delta,F^{M_\delta})$.
It follows that both $E^\Uu_{\delta 0}$ 
and 
$E^\Uu_{\delta 1}$ are strong beyond $\pi_{\delta\mu}(\mu)$.
The other conditions are maintained 
as usual.

Now suppose that $\lh(E_\delta)<\gamma$. Let $\gamma_1=\gamma$ and
$(\gamma_2,\rho_2)=(\tilde{\sigma})_2$. Repeat the subcases, working with $M_\delta|\gamma_1$, 
$\rho_2$, 
$\pi_{\delta 1}$, etc, in place of $M_\delta$, $\rho$, $\pi_{\delta 0}$, etc. Continue in this 
manner until reaching some lift of $E_\delta$. This completes the 
definition of 
$\Uu\rest(\delta+2,1)$. Calculations as above maintain the inductive hypotheses.
\end{case}

This completes the proof for standard iterability. Now suppose that $N_\zeta$ is $z$-solid and 
$\Tt$ on $\core_z(N_\zeta)$ is $z$-maximal. We make the following adjustments to the preceding 
construction to ensure that $\Uu$ is non-overlapping.\footnote{$\Uu$ is padded;
non-overlapping here means that the equivalent non-padded tree is non-overlapping.
That is, for all $(\beta,j)+1<\lh(\Uu)$ and $(\beta',j')=\pred^\Uu((\beta,j)+1)$
and $(\alpha,i)$ with $E^\Uu_{\alpha i}\neq\emptyset\neq E^\Uu_{\beta j}$,
we have $(\beta',j')\leq(\alpha,i)$ iff $\crit(E^\Uu_{\beta j})<\strength^{M^\Uu_{\alpha i}}(E^\Uu_{\alpha i})$.} (The rest of the theorem follows as usual.) 
Things are almost as before; the main difference regards extender 
selection, which we now 
explain. Whether or not  
$E^\Uu_{\alpha i}\neq\emptyset$ is determined as before.
Suppose  $E^\Uu_{\alpha i}\neq\emptyset$. Let $F=F^{M_\alpha|\gamma_{\alpha 
i}}$ and $F^+=F^{Q_{\alpha i}}$.
If $\nu(F)$ is not a limit cardinal of $M_\alpha|\gamma_{\alpha i}$
then let $Q'=Q_{\alpha i}$. If $\nu(F)$ is a limit cardinal of $M_\alpha|\gamma_{\alpha i}$
then let $m=n_{\alpha i}$, let
\[ \nu'=\sup\tau_{\alpha i}^{m 
0}\com\pi_{\alpha i}``\nu(F) \]
(so $\nu'$ is a limit cardinal of $Q_{\alpha i}$ and 
$\nu'\leq\nu(F^+)$), let $F'=F^+\rest\nu'$, and let $Q'\ins 
Q_{\alpha i}$ be such that $F^{Q'}\rest\nu(F^{Q'})=F'$. Because $\nu'$ is a cardinal of $Q_{\alpha 
i}$, there is $\chi\leq\xi_{\alpha i}$ such that
$Q'=N_\chi^{\CC_{\alpha i}}$. Let $E^\Uu_{\alpha i}$ witness \ref{dfn:nice_con} with respect to 
$Q'$, with 
$\strength^{M^\Uu_{\alpha i}}(E^\Uu_{\alpha i})$ small as possible.

With this change, the foregoing proof still goes through essentially as before.
Moreover, we claim that $\Uu$ is 
non-overlapping.
For let $(\alpha,i)$ be such that $E^\Uu_{\alpha i}\neq\emptyset$ and
let $m=n_{\alpha i}$ and $F$ be as above. Our choice of $E^\Uu_{\alpha i}$ implies that
\[ \strength^{M^\Uu_{\alpha 
i}}(E^\Uu_{\alpha i})\leq\sup\tau^{m 0}_{\alpha i}\com\pi_{\alpha 
i}``\nu(F). \]
Using this, the $z$-maximality of $\Tt$, and the agreement condition 
\ref{item:agreement}, it is straightforward to verify that $\Uu$ is non-overlapping.
%***dropped this next footnote out, as it's obvious enough, other than that it doesn't always make sense due to domain problem
%\begin{comment}
%\footnote{If $i<u_\alpha$ and $\nu=\nu(F)$ then $\pi_{\alpha,i+1}(\nu)\geq\nu'=\sup\tau^{m0}_{\alpha i}\com\pi_{\alpha 
%i}``\nu\geq\strength^{M^\Uu_{\alpha i}}(E^\Uu_{\alpha i})$,
%which is enough. (In the former construction we would have 
%had $\pi_{\alpha,i+1}(\nu)>\nu(F^{Q_{\alpha i}})$, but it seems this might fail now).}
%\end{comment}

This completes the proof.
\end{proof}

\bibliographystyle{plain}
\bibliography{resurrection}

\end{document}